\documentclass[a4paper,twoside,10pt]{article}

\newcommand{\mytitle}{Inverse Problems of Combined Photoacoustic and Optical Coherence Tomography}
\title{\mytitle}

\author{P. Elbau$^1$\\{\footnotesize\href{mailto:peter.elbau@univie.ac.at}{peter.elbau@univie.ac.at}}
\and L. Mindrinos$^1$\\{\footnotesize\href{mailto:leonidas.mindrinos@univie.ac.at}{leonidas.mindrinos@univie.ac.at}}
\and O. Scherzer$^{1,2}$\\{\footnotesize\href{mailto:otmar.scherzer@univie.ac.at}{otmar.scherzer@univie.ac.at}}}

\usepackage[utf8]{inputenc}

\usepackage{amsmath}

\usepackage{longtable}

\usepackage{amssymb}

\usepackage{mathrsfs}

\usepackage{mathtools}

\usepackage{siunitx}

\usepackage[a4paper,centering,bindingoffset=0cm,marginpar=0cm]{geometry}

\usepackage{titlesec}
\usepackage[font=footnotesize,format=plain,labelfont=sc,textfont=sl,width=0.75\textwidth,labelsep=period]{caption}
\titleformat{\section}{\filcenter\sc\large}{\thesection.\;}{0em}{}
\titleformat{\subsection}[runin]{\bf}{\thesubsection.\;}{0em}{}[.]

\newpagestyle{headers}%
{%
	\headrule
	\sethead%
		[\thepage]%
		[\footnotesize\sc P.~Elbau, L.~Mindrinos, and O.~Scherzer]%
		[]%
		{}%
		{\footnotesize\sc{\mytitle}}%
		{\thepage}
}
\usepackage[pdftex,colorlinks=true,linkcolor=blue,citecolor=green,urlcolor=blue,bookmarks=true,bookmarksnumbered=true]{hyperref}
\hypersetup
{
    pdfauthor={Peter Elbau, Leonidas Mindrinos, Otmar Scherzer},
    pdfsubject={Quantitative Photoacoustic Tomography and Optical Coherence Tomography},
    pdftitle={\mytitle},
    pdfkeywords={Quantitative Photoacoustic Tomography, Optical Coherence Tomography, Hybrid imaging}
}

\usepackage[hyperref,amsmath,thmmarks]{ntheorem}
\usepackage{aliascnt}
\theoremstyle{break}
\theoremseparator{.}
\theoremheaderfont{\kern-1em\normalfont\bfseries}
\newtheorem{lemma}{Lemma}[section]

\newaliascnt{proposition}{lemma}
\newtheorem{proposition}[proposition]{Proposition}
\aliascntresetthe{proposition}

\newaliascnt{corollary}{lemma}

\aliascntresetthe{corollary}

\newaliascnt{invpro}{lemma}

\aliascntresetthe{invpro}

\theorembodyfont{\normalfont}
\newaliascnt{definition}{lemma}
\newtheorem{definition}[definition]{Definition}
\aliascntresetthe{definition}

\theorembodyfont{\normalfont}
\newaliascnt{example}{lemma}

\aliascntresetthe{example}

\theorembodyfont{\normalfont}
\newaliascnt{convention}{lemma}

\aliascntresetthe{convention}

\theorembodyfont{\normalfont\itshape}
\theoremseparator{:}
\theoremheaderfont{\kern-1em\normalfont\itshape}
\newaliascnt{remark}{lemma}
\newtheorem{remark}[remark]{Remark}
\aliascntresetthe{remark}

\theoremstyle{nonumberplain}
\theorembodyfont{\normalfont}
\theoremsymbol{\ensuremath{\square}}
\newtheorem{proof}{Proof}

    \usepackage[ntheorem]{empheq}

\usepackage{bbm}
\usepackage{bm}
\newcommand{\R}{\mathbbm{R}}

\renewcommand{\b}{\bm}
\renewcommand{\j}{\bm{\bar{j}}}

\newcommand{\e}{\mathrm e}
\renewcommand{\i}{\mathrm i}
\renewcommand{\d}{\,\mathrm d}

\let\RE\Re
\let\Re=\undefined
\DeclareMathOperator{\Re}{\RE e}
\let\IM\Im
\let\Im=\undefined
\DeclareMathOperator{\Im}{\IM m}

\DeclareMathOperator{\curl}{\mathbf{curl}}
\let\div=\undefined
\DeclareMathOperator{\div}{div}
\DeclareMathOperator{\grad}{\mathbf{grad}}
\DeclareMathOperator{\supp}{supp}

\begin{document}

\maketitle
\hspace*{1em}
\parbox[t]{0.49\textwidth}{\footnotesize
\hspace*{-1ex}$^1$Computational Science Center\\
University of Vienna\\
Oskar-Morgenstern-Platz 1\\
A-1090 Vienna, Austria}
\parbox[t]{0.4\textwidth}{\footnotesize
\hspace*{-1ex}$^2$Johann Radon Institute for Computational\\
\hspace*{1em}and Applied Mathematics (RICAM)\\
Altenbergerstra{\ss}e 69\\
A-4040 Linz, Austria}

\vspace*{2em}

\begin{abstract}
Optical coherence tomography (OCT) and photoacoustic tomography (PAT) are emerging non-invasive biological and medical imaging techniques. It is a recent trend in experimental science to design experiments that perform PAT and OCT imaging at once. In this paper we present a mathematical model describing the dual experiment.

Since OCT is mathematically modelled by Maxwell's equations or some simplifications of it, whereas the light propagation in quantitative photoacoustics is modelled by (simplifications of) the radiative transfer equation, the first step in the derivation of a mathematical model of the dual experiment is to obtain a unified mathematical description, which in our case are Maxwell's equations. As a by-product we therefore derive a new mathematical model of photoacoustic tomography based on Maxwell's equations.

It is well known by now, that without additional assumptions on the medium, it is not possible to uniquely reconstruct all optical parameters from either one of these modalities alone. We show that in the combined approach one has additional information, compared to a single modality, and the inverse problem of reconstruction of the optical parameters becomes feasible.
\end{abstract}

\section{Introduction}
Only recently there have been developed experimental setups that can perform photoacoustic and optical coherence tomography experiments in parallel, see~\cite{ZhaPovLauAleHof11} and further developments in~\cite{LiuSchmSanZab13, LiuSchmSanZab14}.
Combined setups can be used for imaging biological tissue up to a depth of approximately~\SI{5}{\milli\metre}. 
They have been used for in vivo studies of human and mouse skins. In the current state of experiments, the two recorded modalities are visualised by superposition after registration, we refer the reader to the review paper~\cite{DreLiuKumKam14}.

In this paper, we derive a mathematical model for quantitative imaging of the multi-modal experiment based on Maxwell's equations. This study characterises the additional information obtainable by the dual experiment. In this paper we assume that the sample is an  inhomogeneous, isotropic, non-magnetic, linear dielectric medium, and the imaging parameters of the medium are the conductivity, the susceptibility and the Gr\"uneisen parameter. 

Even though the mathematical modelling of optical coherence tomography based on Maxwell's equations is well established, most of the literature relies on simplified models of the Helmholtz equation \cite{AndThrYurTycJorFro04, BruCha05, Fer10, FerDreHitLas03, RalMarCarBop06}, where, as a consequence,  the frequency dependence of the susceptibility is neglected. It is a recent trend to consider the more general case of Maxwell's equations for the mathematical description of the OCT system \cite{BreReiKie15, ElbMinSch15, SanAraBarCarCor15, SilCor13}.

The mathematical modelling of quantitative photoacoustics, on the other hand, is based on the radiative transfer equation or its diffusion approximation~\cite{BalUhl10, HalNeuRab15}. Maxwell's equations have only been considered for mathematical models in thermoacoustic tomography where low-frequency radiation is used for illumination \cite{BalRenUhlZho11, BalZho14}.

To formulate the dual-modal setup, we consider Maxwell's equations as the basic modeling equations since both imaging techniques 
rely on the same excitation. In particular, this provides a more general model for quantitative photoacoustics based on Maxwell's equations, which is also applicable in the case of high frequency radiation.

This paper is organized as follows: In \autoref{sePAT}, we give an overview of the inverse problem in photoacoustics. We present the governing equations and the commonly used simplifications. Next, we formulate mathematically the problem of optical coherence tomography starting from Maxwell's microscopic equations. For an extensive mathematical modelling of optical coherence tomography, we refer to~\cite{ElbMinSch15}. The system of equations for the dual experiment is obtained in \autoref{sePAT_OCT}. In \autoref{seFre}, the equation system is transformed to a Fredholm integral equation for the Gr\"uneisen parameter under the far field and Born approximations, 
respectively. Finally, we discuss the connection between the optical parameters appearing in the radiative transfer equation and Maxwell's equations.

To simplify the reading, we summarise the basic notations used in this paper in the following tables.
\begin{longtable}[c]{|l|l|c|}
\hline
Symbol & Quantity & Definition \\ \hline
$c$ & speed of light & \eqref{eq:rtf0} \\
 $\mu_{\mathrm a}$ & absorption coefficient & \eqref{eq:rtf0}\\
 $\mu_{\mathrm s}$ & scattering coefficient & \eqref{eq:rtf0}\\
  $h$ & Planck constant & \eqref{eq:fluence}\\ 
 $\epsilon$ & diffusion coefficient & \eqref{eqTau}\\
 $\gamma$ & Gr\"uneisen parameter & \eqref{eqInitialPressure}\\
 $K$ & Bulk modulus & \eqref{eq:wave}\\
  $\varrho$ & mass density & \eqref{eq:wave}\\
  $c_s$ & speed of sound & \eqref{eq:wave2} \\
   $\mu$ & electric permittivity & \eqref{eqDielectricum}\\
 $\sigma$ & electric conductivity & \eqref{eqSusc}\\
 $\chi$ & electric susceptibility & \eqref{eqSusc}\\
 \hline
\caption{Physical Quantities}
\end{longtable}

\section{Photoacoustic Imaging}\label{sePAT}
Photoacoustic tomography is a hybrid imaging technique which measures the acoustic response of an object upon illumination with an electromagnetic wave. In experiments the object is illuminated with a short laser pulse, typically with light in the near infrared 
spectrum. This pulse is absorbed by the medium, where the exact amount varies locally depending on the material properties, in particular on the absorption coefficient; and thereby a fraction is transformed into heat energy. 
This leads to a local change of pressure in the medium, where the transformation factor is again depending on the local thermodynamic properties of the medium, typically reduced to the single Gr\"uneisen parameter. Finally, this pressure propagates as an ultrasonic wave through the medium and is recorded at every point on a surface around the object as a function of time.

For a mathematical description of this measurement setup, we have to model the three different phenomena:
\begin{enumerate}
\item
The propagation of the laser light in the medium,
\item
the transformation of the absorbed energy into a pressure distribution, and
\item
the propagation of the pressure wave inside the medium.
\end{enumerate}

For the readers convenience we summarize the variables, which are used throughout this section, as a backup reference:
\begin{longtable}[c]{|l|l|c|}
\hline
Symbol & Quantity & Definition \\ \hline
 $\psi_{\bm \vartheta}$ & photon density & \eqref{eq:rtf0}\\
  $\Theta$ & phase function & \eqref{eq:rtf0}\\
 $\Phi_{\bm\vartheta}$ & energy fluence & \eqref{eq:fluence}\\
  $\bar \Phi$ & total energy fluence & \eqref{eq:tef}\\
  \hline
\caption{Radiative Transfer Equation}
\end{longtable}

The light propagation is usually modelled via the radiative transfer equation. Assuming that we have a laser pulse of fixed frequency $\nu>0$, the photon density $\psi_{\bm\vartheta}(t,\b x)$ at time $t\in\R$ and coordinate $\b x \in \R^3$ moving in the direction $\bm\vartheta\in \mathbbm{S}^2$ fulfils the equation
\begin{equation}
 \label{eq:rtf0}
 \frac1c\partial_t\psi_{\bm\vartheta}(t,\b x)+\left<\bm\vartheta,\nabla_{\b x}\psi_{\bm\vartheta}(t,\b x)\right>+(\mu_{\mathrm s}(\b x)+\mu_{\mathrm a}(\b x))\psi_{\bm\vartheta}(t,\b x) =\frac{\mu_{\mathrm s}(\b x)}{4\pi}\int_{\mathbbm{S}^2}\Theta(\b x,\bm{\tilde{\vartheta}},\bm\vartheta)\psi_{\bm\tilde{\vartheta}}(t,\b x)\d s(\bm{\tilde{\vartheta}}), 
\end{equation}
where $\mu_{\mathrm a}$ is the absorption coefficient and $\mu_{\mathrm s}$ is the scattering coefficient at the radiation 
frequency $\nu$, see for example~\cite{WanWu07}. Here, the constant $c$ denotes the speed of light and $\Theta$ is the phase function, that is $\Theta(\b x,\bm{\tilde{\vartheta}},\bm\vartheta)$ is the probability density that a photon heading into a direction $\bm{\tilde{\vartheta}}\in \mathbbm{S}^2$ is scattered at the position $\b x\in\R^3$ into the direction $\bm\vartheta\in \mathbbm{S}^2$.

Since in photoacoustic imaging the laser excitation occurs to be very short and the absorption of the light happens much faster 
than the propagation of the acoustic wave, the light distribution as a function of time is not relevant, but only the total energy 
absorbed at each point matters. Thus we turn our attention to the total energy fluence $\Phi_{\bm\vartheta}$ originating from photons 
moving in the direction $\bm\vartheta\in \mathbbm{S}^2$ as new variable:
\begin{equation}
\label{eq:fluence}
 \Phi_{\bm\vartheta}(\b x) = h\nu\int_{-\infty}^\infty\psi_{\bm\vartheta}(t,\b x)c\d t,\quad \b x\in\R^3,\;\bm\vartheta\in \mathbbm{S}^2, 
\end{equation}
where $h$ denotes the Planck constant. Then, $\Phi_{\bm\vartheta}$ solves the equation
\begin{equation}
\label{eq:rte}
\left<\bm\vartheta,\nabla_{\b x}\Phi_{\bm\vartheta}(t,\b x)\right>+(\mu_{\mathrm s}(\b x)+\mu_{\mathrm a}(\b x))\Phi_{\bm\vartheta}(t,\b x) = \frac{\mu_{\mathrm s}(\b x)}{4\pi}\int_{\mathbbm{S}^2}\Theta(\b x,\bm{\tilde{\vartheta}},\bm\vartheta)\Phi_{\bm\tilde{\vartheta}}(t,\b x)\d s(\bm{\tilde{\vartheta}}).
\end{equation}
However, since the phase function $\Theta$ is unknown, often the diffusion approximation
\[ \div(\epsilon\nabla\bar\Phi)(\b x)=\mu_{\mathrm a}(\b x)\bar\Phi(\b x), \]
with the diffusion coefficient $\epsilon$ given by 
\begin{equation}\label{eqTau}
\epsilon(\b x)=[3(\mu_{\mathrm a}(\b x)+(1-\tfrac13\Theta_1(\b x))\mu_{\mathrm s}(\b x)]^{-1},
\end{equation}
is used, which is an equation for the total energy fluence
\begin{equation}
\label{eq:tef}
\bar\Phi(\b x) = \frac1{4\pi}\int_{\mathbbm{S}^2}\Phi_{\bm\vartheta}(\bm x)\d s(\bm\vartheta).
\end{equation}
This is justified if $\Theta(\bm x,\bm{\tilde{\vartheta}},\bm\vartheta)\approx1+\Theta_1(x)\langle \bm{\tilde{\vartheta}},\bm\vartheta\rangle$ and $\Phi_{\bm\vartheta}(\bm x)\approx\bar\Phi(\b x)+\left<\bm\phi_1(\bm x),\bm\vartheta\right>$ for some functions $\Theta_1$ and $\bm\phi_1$, which is a good approximation if we have a strongly scattering medium, see \cite[Chapter XXI, \S5]{DauLio88f}.

The absorbed energy at a point $\b x$ is then given by $\mu_{\mathrm a}(\b x)\bar\Phi(\b x)$. Moreover, it is common to assume that the produced pressure density $p^{(0)}$ inside the medium is proportional to the absorbed energy where the proportionality coefficient is the Gr\"uneisen parameter $\gamma$:
\begin{equation}\label{eqInitialPressure}
p^{(0)}(\b x) = \gamma(\b x)\mu_{\mathrm a}(\b x)\bar\Phi(\b x).
\end{equation}

For the propagation of the acoustic wave, the standard assumption is that we have an elastic medium with mass density $\varrho$, bulk modulus $K$, and vanishing shear modulus. Then, according to linear elasticity theory, see for instance \cite{ColKre98}, we obtain the equation
\begin{equation}
\label{eq:wave}
 \partial_{tt}p(t,\b x) = K(\b x)\div_{\b x} (\tfrac1\varrho\nabla_xp)(t,\b x).
\end{equation}
Moreover, the density is usually considered to be approximatively constant throughout the medium, so that we may simplify this equation to the linear wave equation
\begin{equation}
\label{eq:wave2}
\partial_{tt}p(t,\b x) = c_{\mathrm s}^2(\b x)\Delta_xp(t,\b x)
\end{equation}
with the speed of sound $c_{\mathrm s}=\sqrt{\frac K\varrho}$.

Finally, we acquire the measurements
\begin{equation}
 \label{eq:g}
 g(t,\bm\xi) = p(t,\bm\xi),\quad t>0,\;\bm\xi\in\mathcal D 
\end{equation}
at some detector surface $\mathcal D\subset\R^3$.

Putting all this together, the inverse problem of photoacoustic imaging is to recover the material parameters $\mu_{\mathrm a}$, $\mu_{\mathrm s}$ (which in this approximation only enters via the diffusion coefficient $\epsilon$, see~\eqref{eqTau}), $\gamma$, and $c_{\mathrm s}$ from the given measurement data $g$.

This problem can be solved in two steps. First, we consider the acoustic part, given by the equation system
\begin{equation}
 \label{eq:wave3}
 \boxed{
 \begin{aligned}
\partial_{tt}p(t,\b x) &= c_{\mathrm s}^2(\b x)\Delta_xp(t,\b x),\qquad&&t>0,\;\b x\in\R^3, \\
\partial_tp(0,\b x)&=0,\qquad&&\b x\in\R^3, \\
p(t,\bm\xi)&=g(t,\bm\xi),\qquad&&t>0,\;\bm\xi\in\mathcal D, \\
p(0,\b x) &= p^{(0)}(\b x),\qquad&&\b x\in\R^3.
\end{aligned}}
\end{equation}
For given, non-trapping speed of sound $c_{\mathrm s}$ (in practice, a common assumption is that $c_{\mathrm s}$ is approximatively constant), the measurements $g$ allow to uniquely reconstruct the initial pressure~$p^{(0)}$, see \cite{AgrKuc07}. If the detector surface has a simple geometry (the easiest examples are a planar or a spherical detector) and the speed of sound is assumed to be constant, then there are a lot of explicit reconstruction formulas for $p^{(0)}$ available: \cite{FinHalRak07,FinPatRak04,Kun07b,Kun07,Kun11a,Nat11,Pal12,Pal14,XuWan02a,XuWan05,XuFenWan02,XuXuWan02}.

For unknown speed of sound, it was suggested in \cite{KirSch12} to use multiple photoacoustic measurements with a focussed illumination, see also \cite{ElbSch15,ElbSchSchu12} for reconstruction formulas for this kind of illumination. For a more extensive review of these works, we refer to \cite{KucSch15}.

So, let us assume that we can invert the acoustic problem and recover $p^{(0)}$ from the photoacoustic measurements. Then, as the second step, it remains the problem of reconstructing the material parameters $\mu_{\mathrm a}$, $\mu_{\mathrm s}$, $\gamma$ from the internal data $p^{(0)}$ via the relations
\begin{equation*}
\boxed{
\begin{aligned}
\div_{\b x}(\epsilon\nabla\bar\Phi)(\b x)&=\mu_{\mathrm a}(\b x)\bar\Phi(\b x),\qquad&&\b x\in\R^3, \\
\bar\Phi(\bm\xi) &= \bar\Phi^{(0)}(\bm\xi),\qquad&&\bm\xi\in\mathcal D, \\
p^{(0)}(\b x) &= \gamma(\b x)\mu_{\mathrm a}(\b x)\bar\Phi(\b x),\qquad&&\b x\in\R^3,
\end{aligned}}
\end{equation*}
where we assumed that we know the boundary data $\bar\Phi^{(0)}$ at the detector surface, since we are expected to know all the material parameters outside the object, in particular in the vicinity of the detector, and thus $p^{(0)}$ allows us to calculate $\bar\Phi$ around the detector surface. However, in practice, these data are not used, because the absorption is considered very small outside the object 
of interest. We are only using these data mathematically to be able to define the modelling equations in all $\R^3$, which is state of the art in this field.

In \cite{BalRen11a}, it was shown that even with multiple photoacoustic measurements, it is not possible to reconstruct all three parameters, but from the knowledge of one of the parameters (usually, the Gr\"uneisen parameter $\gamma$ is assumed to be approximatively constant, although there is no physical evidence for this) the other two can be explicitly calculated, see also \cite{Bal12} for a review about the reconstruction from internal data.

Thus, the inverse photoacoustic problem is not well-posed, as it requires an additional three dimensional data to allow the reconstruction of all material parameters.

\section{Optical Coherence Tomography} \label{sec:oct}
Opposed to photoacoustic imaging, optical coherence tomography is not a hybrid imaging technique. It visualises the back-scattered light of an object. However, analogously to photoacoustic imaging, the object is illuminated with a short laser pulse in the visible or near infrared spectrum.
In practical realisations, the illumination is triggered to produce a plane wave, and the scattered wave is measured at a 
planar detector parallel to the plane wave illumination.

Since the time resolution of the detectors, which would be required for the desired spatial resolution, is practically hard to achieve, the detection of the reflected wave is accomplished by superimposing it with a reference beam which is a copy of the wave used for illumination (produced by a beam splitter from the original wave) sent to a mirror instead of to the object. Then, the total intensity of this superposition is measured at every point on the detection plane.

For the readers convenience we summarize the variables, which are used throughout this section, as a backup reference:
\begin{longtable}[c]{|l|l|c|}
\hline
Symbol & Quantity & Definition \\ \hline
$\bm e$ & microscopic electric field & \eqref{eqMaxwell1}\\
 $\rho$ & microscopic charge density & \eqref{eqMaxwell1}\\
 $\bm b$ & microscopic magnetic field & \eqref{eqMaxwell2}\\
 $\bm j$ & microscopic electric current & \eqref{eqMaxwell4}\\
 $\bm E$ & averaged electric field & \eqref{eqMaxwellMacro1}\\
  $\bar \rho$ & averaged charge density & \eqref{eqMaxwellMacro1}\\
 $\bm B$ & averaged magnetic field & \eqref{eqMaxwellMacro2}\\
 $\bar{\bm j}$ & averaged electric current & \eqref{eqMaxwellMacro4}\\
  $\tilde \rho$ & approximation of $\bar \rho$ & \eqref{eqq_rho1}\\
 $\bm J$ & approximation of $\bar{\bm j}$ & \eqref{eqq_rho2}\\
  $\bm P$ & electric polarization  & \eqref{eqPolarisationElectric}\\
   $\bm M$ & magnetic polarization  & \eqref{eqPolarisationMagnetic}\\
    ${\bm E}_s$ & stationary electric field & \eqref{eq:Maxwells}\\
 ${\bm B}_s$ & stationary magnetic field & \eqref{eq:Maxwells}\\
 ${\bm j}_s$ & stationary electric current & \eqref{eq:Maxwells}\\
 $\rho_s$ & stationary charge density & \eqref{eq:Maxwells}\\
  $\mathcal E$ & electrostatic potential energy & \eqref{eq:me}\\
 $\mathcal W$ & total electromagnetic energy density & \eqref{eq:ws}\\
 $\b S$ & Poynting vector & \eqref{eq:ws}\\
 $\b D$ & electric displacement & \eqref{eq:electric_displacement}\\
 \hline
\caption{Electromagnetic Waves}
\end{longtable}
\subsection{Microscopic Maxwell's equations}

The interaction of the laser pulse, which is an electromagnetic wave, and the medium is best described by Maxwell's microscopic equations. 
These are equations for the time evolution of the microscopic electric and magnetic fields $\b e$ and $\b b$ for given microscopic charge density $\rho$ and microscopic electric current $\b j$:
\begin{subequations}\label{eqMaxwell}
\begin{alignat}{2}
\div_{\b x} \b e(t,\b x) &= 4\pi\rho(t,\b x),\quad &&t\in\R,\;\b x\in\R^3,\label{eqMaxwell1} \\
\div_{\b x} \b b(t,\b x) &= 0,\quad &&t\in\R,\;\b x\in\R^3,\label{eqMaxwell2} \\
\curl_{\b x} \b e(t,\b x) &= -\frac1c\partial_t \b b(t,\b x),\quad &&t\in\R,\;\b x\in\R^3, \label{eqMaxwell3} \\
\curl_{\b x} \b b(t,\b x) &= \frac1c\partial_t \b e(t,\b x)+\frac{4\pi}c \b j(t,\b x),\quad &&t\in\R,\;\b x\in\R^3. \label{eqMaxwell4}
\end{alignat}
\end{subequations}
\begin{remark}
If the initial data $\b e(0,\cdot)$ and $\b b(0,\cdot)$ are compatible (that is, they are satisfying the equations 
\eqref{eqMaxwell1} and \eqref{eqMaxwell2} at $t=0$) and the continuity equation, 
\begin{equation}\label{eqChargeConservation}
\partial_t\rho(t,\b x)+\div_{\b x}\b j(t,\b x) = 0,\quad t\in\R,\,\b x\in\R^3,
\end{equation}
holds, then the equations \eqref{eqMaxwell3} and \eqref{eqMaxwell4} imply \eqref{eqMaxwell1} and \eqref{eqMaxwell2}. 

To see this, take the divergence of the two vector equations \eqref{eqMaxwell3} and \eqref{eqMaxwell4}. Then, using the identity $\div_{\b x} \curl_{\b x} \b f = 0$ for every vector field $\b f$, it follows with \eqref{eqChargeConservation} that 
\[ \div_{\b x} \partial_t \b b = 0\quad\text{and}\quad\div_{\b x} \partial_t \b e = -4 \pi \div_{\b x} \b j = 4\pi\partial_t\rho. \]
Integrating these equations over time gives the equations \eqref{eqMaxwell1} and \eqref{eqMaxwell2}.
\end{remark}

To obtain a formula for the evolution of the charge density $\rho$ and the electric current $\b j$, we consider a medium which consists of
\begin{itemize}
 \item $N$ charged particles
 \item  at the positions $\b x_i:\R\to\R^3$, $i=1,\ldots,N$, as functions of time
 \item with masses $m_i$ and 
 \item charge densities $q_i\rho_i$, where $q_i\in\R$ and $\rho_i\in C^\infty_{\mathrm c}(\R^3)$ are non-negative functions with $\|\rho_i\|_{L^1}=1$. Because these particles have physical radii smaller than \SI{e-14}{\metre}, $\rho_i$ can be considered an approximation of a $\delta$-distribution.
\end{itemize} 
Then, the charge density $\rho$ and the electric current $\b j$ are given by
\begin{subequations}\label{eqChargeAndCurrent}
\begin{align}
\rho(t,\b x) &= \sum_{i=1}^Nq_i\rho_i(\b x-\b x_i(t)),\quad\text{and}\label{eqChargeDensity} \\
\b j(t,\b x)&=\sum_{i=1}^Nq_i \b x_i'(t)\rho_i(\b x-\b x_i(t)).\label{eqCurrent}
\end{align}
\end{subequations}
Clearly, they fulfil the continuity equation \eqref{eqChargeConservation}.

For the motion of the particles, we assume that they move according to the Lorentz force:
\begin{equation}\label{eqLorentzForce}
m_i \b x_i''(t) = q_i\int_{\R^3}\b e(t,\b x)\rho_i(\b x-\b x_i(t))\d \b x+q_i\frac{\b x_i'(t)}c\times\int_{\R^3}\b b(t,\b x)\rho_i(\b x-\b x_i(t))\d \b x.
\end{equation}
The equations \eqref{eqMaxwell}, \eqref{eqChargeAndCurrent}, and \eqref{eqLorentzForce} completely define the evolution of the medium starting from initial conditions $\b x_i(0)$ and $\b x_i'(0)$, $i=1,\ldots,N$ for the particles and $\b e(0,\b x)$ and $\b b(0,\b x)$, $\b x\in\R^3$, for the microscopic electromagnetic fields, see for example \cite[Chapter 6]{Jac98}.

\subsection{Macroscopic Maxwell's equations}

However, solving the huge system of equations \eqref{eqLorentzForce} is practically impossible.
Therefore, we average the fields with respect to the space variable by some non-negative weighting 
function $w\in C^\infty_{\mathrm c}(\R^3)$ with support around $0$ and $\|w\|_{L^1}=1$:
\begin{align*}
\b E(t,\b x) &= \int_{\R^3}w(\b y-\b x)\b e(t,\b y)\d \b y, & \b B(t,\b x) &= \int_{\R^3}w(\b y-\b x)\b b(t,\b y)\d \b y, \\
\j(t,\b x) &= \int_{\R^3}w(\b y-\b x)\b j(t,\b y)\d \b y,&\bar\rho(t,\b x) &= \int_{\R^3}w(\b y-\b x)\rho(t,\b y)\d \b y.
\end{align*}
Then, taking the average of the equation system \eqref{eqMaxwell} with respect to the function $w$, we obtain Maxwell's macroscopic equations
\begin{subequations}\label{eqMaxwellMacro}
\begin{alignat}{2}
\div_{\b x} \b E(t,\b x) &= 4\pi \bar\rho (t,\b x),\quad &&t\in\R,\;\b x\in\R^3,\label{eqMaxwellMacro1} \\
\div_{\b x} \b B(t,\b x) &= 0,\quad &&t\in\R,\;\b x\in\R^3,\label{eqMaxwellMacro2} \\
\curl_{\b x} \b E(t,\b x) &= -\frac1c\partial_t \b B(t,\b x),\quad &&t\in\R,\;\b x\in\R^3,\label{eqMaxwellMacro3} \\
\curl_{\b x} \b B(t,\b x) &= \frac1c\partial_t \b E(t,\b x)+\frac{4\pi}c\j (t,\b x),\quad &&t\in\R,\;\b x\in\R^3,\label{eqMaxwellMacro4}
\end{alignat}
\end{subequations}
for the averaged electromagnetic fields $\b E$ and $\b B$ depending on the averaged charge density $\bar\rho$ and the averaged electric current $\j$.

\begin{remark}\label{rePhysicalNotation}
Since the medium usually consists of molecules, most of the particles are clustered around the centers of these molecules. Let us relabel the particles $(\b x_i)_{i=1}^N$ therefore as $((\b x_{i,k})_{i=1}^{I_k})_{k=1}^K$ such that $\b x_{i,k}$ is the $i$-th particle of the $k$-th cluster of particles, $i=1,\ldots,I_k$, $k=1,\ldots,K$, write~$q_{i,k}$ for its electric charge, and let $\bar{\b x}_k$ be the mass center of the $k$-th molecule. 

Then, we introduce the zeroth order approximations $\tilde\rho$ and $\b J$ of the charge density $\bar\rho$ and the electric current $\j$, respectively, which are obtained from the averaged fields $\bar\rho$ and $\j$ when the molecules are assumed to be point like, that is, setting $\b x_{i,k}=\bar{\b x}_k$ for all $i$:
\begin{subequations}\label{eqq_rho}
\begin{align}
\tilde\rho(t,\b x) &= \sum_{k=1}^Kw(\bar{\b x}_k(t)-\b x)\sum_{i=1}^{I_k}q_{i,k},\label{eqq_rho1} \\
\b J(t,\b x) &= \sum_{k=1}^K\bar{\b x}_k'(t)w(\bar{\b x}_k(t)-\b x)\sum_{i=1}^{I_k}q_{i,k}.\label{eqq_rho2}
\end{align}
\end{subequations}
We note that $\tilde\rho$ and $\b J$ also fulfil the continuity equation
\begin{equation}\label{eqContinuityZerothOrd}
\partial_t\tilde\rho(t,\b x)+\div_x\b J(t,\b x)=0.
\end{equation}
We now write the deviation $\bar\rho-\tilde\rho$ of the averaged charge density from its zeroth order approximation~$\tilde\rho$ 
as the divergence of some function $\b P:\R\times\R^3\to\R^3$, the so-called electric polarisation:
\begin{equation}\label{eqPolarisationElectric}
\bar\rho(t,\b x) = \tilde\rho(t,\b x)-\div_{\b x} \b P(t,\b x).
\end{equation}
According to Helmholtz's theorem this is always possible, but defines $\b P$ only up to the addition of a divergence-free vector field.

Then, we realise that we have with the continuity equations \eqref{eqChargeConservation} and \eqref{eqContinuityZerothOrd} and the definition \eqref{eqPolarisationElectric} 
\[ \div_{\b x}(\j-\b J-\partial_t \b P) = -\partial_t(\bar\rho-\tilde\rho+\div_{\b x}\b P) = 0. \]
Doing the Helmholtz decomposition of the vector field $\j-\b J-\partial_t \b P$, we therefore find a function $\b M:\R\times\R^3\to\R^3$, the magnetic polarisation, 
(uniquely defined up to a gradient vector field) so that
\begin{equation}\label{eqPolarisationMagnetic}
\j(t,\b x) = \b J(t,\b x)+\partial_t \b P(t,\b x)+c\curl_{\b x} \b M(t,\b x).
\end{equation}

These decompositions \eqref{eqPolarisationElectric} and \eqref{eqPolarisationMagnetic} of $\bar\rho$ and $\j$ into macroscopic ($\tilde\rho$ and $\b J$) and polarisation ($\b P$ and $\b M$) parts are standard in the physics literature to formulate Maxwell's macroscopic equations, see for example \cite[Chapter 6.6]{Jac98}.
\end{remark}

Clearly, the equation system \eqref{eqMaxwellMacro} is not enough to determine the averaged fields $\b E$ and $\b B$, since we would still require a solution of the microscopic problem to obtain the averaged quantities $\bar\rho$ and $\j$. The common way around this is to impose heuristic relations between these functions and the fields $\b E$ and $\b B$. A common choice for biological tissue is a linear relationship.
\begin{definition}\label{deDielectricum}
We call a material where
\begin{equation}\label{eqDielectricum}
\j(t,\b x) = \int_{-\infty}^\infty\mu(\tau,\b x)\b E(t-\tau,\b x)\d\tau
\end{equation}
for some function $\mu\in C^\infty_{\mathrm c}(\R\times\R^3)$ with $\mu(\tau,\b x)=0$ for all $\tau<0$, $\b x\in\R^3$ a non-magnetic, isotropic, linear dielectric medium.
\end{definition}

For a non-magnetic, isotropic, linear dielectric medium, we may solve the equation system \eqref{eqMaxwellMacro} for given parameter $\mu$ for $\b E$ and $\b B$.

\begin{remark}\label{remark34}
In the standard notation introduced in \autoref{rePhysicalNotation}, see in particular \eqref{eqPolarisationElectric} and~\eqref{eqPolarisationMagnetic}, the assumptions corresponding to \eqref{eqDielectricum} are such that no magnetic polarisation exists in the medium, that is $\b M=0$, and that the macroscopic current $\b J$ and the electric polarisation $\b P$ are of the form
\begin{equation}\label{eqSusc}
 \begin{aligned}
  \b J(t,\b x) &= \int_{-\infty}^\infty\sigma(\tau ,\b x)\b E(t-\tau,\b x)\d \tau, \\
\b P(t,\b x) &= \int_{-\infty}^\infty\chi(\tau ,\b x)\b E(t-\tau,\b x)\d \tau,
 \end{aligned}
\end{equation}
where $\sigma\in C^\infty_{\mathrm c}(\R\times\R^3)$ is called the conductivity and $\chi\in C^\infty_{\mathrm c}(\R\times\R^3)$ is the electric susceptibility.

By \eqref{eqPolarisationMagnetic}, the relation between the function $\mu$, defined in \eqref{eqDielectricum}, and the pair $\sigma$ and $\chi$ is simply
\[ \mu(t,\b x) = \sigma(t,\b x)+\partial_t\chi(t,\b x). \]
\end{remark}

\subsection{Problem formulation}

Initially, the electromagnetic fields should describe a laser pulse which did not yet interact with the medium. Thus, we want to look for a solution $(\b E,\b B)$ which for $t<0$ is the superposition of a solution $(\b E^{(0)},\b B^{(0)})$ of the vacuum equations (that is a solution of \eqref{eqMaxwellMacro} with $\bar\rho=0$ and $\j=0$) with support outside the medium (that is outside the support of $\mu(\tau,\cdot)$ for every $\tau\in\R$) and a stationary solution $(\b E_{\textrm s},\b B_{\textrm s})$ produced by the undisturbed medium (we want to assume that without exterior influence, the electric and magnetic fields generated by the medium do not vary in time).

\begin{proposition}\label{thElectricFieldSolMacro}
Let $\Omega\subset\R^3$ be an open set. For given functions $\rho_{\textrm s}\in C^\infty_{\mathrm c}(\R^3)$ and $\b j_{\textrm s}\in C^\infty_{\mathrm c}(\R^3;\R^3)$ with $\supp\rho_{\textrm s}\subset\Omega$, $\supp \b j_{\textrm s}\subset\Omega$, and $\div\b j_{\textrm s}=0$, we define the stationary electromagnetic fields $\b E_{\textrm s},\b B_{\textrm s}\in C^\infty(\R^3;\R^3)$ by
\begin{equation}\label{eq:Maxwells}
\begin{aligned}
\div\b E_{\textrm s} (\b x) &= 4\pi\rho_{\textrm s} (\b x),\qquad &\div\b B_{\textrm s} (\b x) &= 0, \\
\curl\b E_{\textrm s} (\b x) &= 0, \qquad &\curl \b B_{\textrm s} (\b x) &= \frac{4\pi}c\b j_{\textrm s} (\b x).
\end{aligned}\end{equation}

Moreover, let $\b E^{(0)}$ be a solution of the wave equation
\begin{equation}\label{eqWaveInc}
\frac1{c^2}\partial_{tt}\b E^{(0)}(t,\b x) = \Delta_{\b x} \b E^{(0)}(t,\b x),\qquad t\in\R,\;\b x\in\R^3,
\end{equation}
with $\div_x\b E^{(0)}=0$, $\b E^{(0)} (\cdot,\b x)\in L^1(\R)$ for every $x\in\R^3$, and fulfilling that
\begin{equation}\label{eqIntersect}
\supp \b E^{(0)}(t,\cdot) \cap \Omega = \emptyset \text{ for every }t\le0;
\end{equation}

Finally, let $\mu\in C^\infty_{\mathrm c}(\R\times\R^3)$ fulfil $\supp\mu(t,\cdot)\subset\Omega$ for all $t\in\R$, $\mu(t,\b x)=0$ for all $t\le0$, $\b x\in\R^3$, and
\begin{equation}\label{eqStationaryCurrent}
\j_{\textrm s}(\b x)=\b E_{\textrm s}(\b x)\int_0^\infty\mu(t,\b x)\d t\quad\text{for all}\quad\b x\in\R^3.
\end{equation}

Then, there exists a solution $\b E$, $\b B$ of Maxwell's macroscopic equations \eqref{eqMaxwellMacro} for a non-magnetic, isotropic, linear dielectric medium with the parameter $\mu$, as in \autoref{deDielectricum}, and with charge density $\bar\rho(t,\b x)=\rho_{\textrm s}(\b x)-\int_{-\infty}^t\div_{\b x}\bar j(\tau,\b x)\d\tau$ such that we have for all $t\le0$
\begin{align}
\b E(t,\b x) &= \b E^{(0)}(t,\b x)+\b E_{\textrm s}(\b x) \label{eqElectricFieldInitial} \\
\b B(t,\b x)&=-c\int_{-\infty}^t\curl_{\b x}\b E^{(0)}(\tau,\b x)\d\tau+\b B_{\textrm s}(\b x). \label{eqMagneticFieldInitial}
\end{align}

In particular, $\b E$ is the solution of the differential equation
\begin{equation}\label{eqElectricField}
\frac1{c^2}\partial_{tt} \b E(t,\b x)+\curl_{\b x}\curl_{\b x} \b E(t,\b x) = -\frac{4\pi}{c^2}\int_{-\infty}^\infty\partial_t\mu(\tau,\b x)\b E(t-\tau,\b x)\d\tau
\end{equation}
together with the initial condition \eqref{eqElectricFieldInitial}, which is equivalent to the integral equation
\begin{equation}\label{eqElectricFieldMacroIntegral}
\b E(t,\b x) = \b E^{(0)}(t,\b x) + \b E_{\textrm s} (\b x)-\int_{\R^3}\int_{-\infty}^{t-\frac{|\b y|}c}\frac1{|\b y|}\left(\grad_{\b x}\div_{\b x}-\frac1{c^2}\partial_{tt}\right)\j(\tau,\b x+\b y)\d\tau\d \b y,
\end{equation}
where $\j$ is given by \eqref{eqDielectricum}.
\end{proposition}
\begin{proof}
We first check that \eqref{eqElectricFieldInitial} and \eqref{eqMagneticFieldInitial} solve Maxwell's equations \eqref{eqMaxwellMacro} for $t\le 0$. We have directly by construction
\[ \div_{\b x}\b E(t,\b x) = 4\pi\rho_{\textrm s}(\b x)\quad\text{and}\quad\div_{\b x}\b B(t,\b x) = 0, \]
which are the equations \eqref{eqMaxwellMacro1} and \eqref{eqMaxwellMacro2}. Moreover, we have \eqref{eqMaxwellMacro3}, since
\[ \partial_t\b B(t,\b x) = -c \curl_{\b x}\b E^{(0)}(t,\b x) = -c \curl_{\b x}\b E(t,\b x). \]
Finally, we get
\[\frac1c \partial_t\b E(t,\b x) = \frac1c \int_{-\infty}^t\partial_{tt}\b E^{(0)}(\tau,\b x)\d\tau = c \int_{-\infty}^t\Delta_{\b x}\b E^{(0)}(\tau,\b x)\d\tau. \]
Since $\div_{\b x}\b E^{(0)}=0$, we get with the identity $\curl_{\b x}\curl_{\b x} \b F = \grad_{\b x} \div_{\b x} \b F - \Delta_{\b x} \b F$ for every function $\b F\in C^2(\R^3;\R^3)$ that
\begin{align*}
\frac1c \partial_t\b E(t,\b x) &= -c\int_{-\infty}^t\curl_{\b x}\curl_{\b x}\b E^{(0)}(\tau,\b x)\d\tau \\
&= \curl_{\b x}\b B(t,\b x)-\curl_{\b x}\b B_{\textrm s}(\b x) = \curl_{\b x}\b B(t,\b x)-\frac{4\pi}c\b j_{\textrm s}(\b x).
\end{align*}
It remains to check that the current $\j$, defined by relation \eqref{eqDielectricum}, indeed fulfils $\j(t,x)=\b j_{\textrm s}(\b x)$ for $t\le0$. Using that $\mu(t,\b x)=0$ for $t\le0$, condition \eqref{eqIntersect}, and \eqref{eqStationaryCurrent}, we find
\begin{align*}
\int_{-\infty}^\infty\mu(t-\tau,\b x)\b E(\tau,\b x)\d\tau &= \int_{-\infty}^0\mu(t-\tau,\b x)(\b E^{(0)}(t,\b x)+\b E_{\textrm s}(\b x))\d\tau \\
&= \b E_{\textrm s}(\b x)\int_{-\infty}^0\mu(t-\tau,\b x)\d\tau = \b j_{\textrm s}(\b x).
\end{align*}
Thus, we have shown that \eqref{eqElectricFieldInitial} and \eqref{eqMagneticFieldInitial} solve Maxwell's equations \eqref{eqMaxwellMacro} for $t\le 0$.

Combining Maxwell's macroscopic equations \eqref{eqMaxwellMacro3} and \eqref{eqMaxwellMacro4}, 
it follows that $\b E$ in particular solves the equation
\[ \frac1{c^2}\partial_{tt}\b E(t,\b x)+\curl_{\b x}\curl_{\b x} \b E(t,\b x) = -\frac{4\pi}{c^2}\partial_t \j (t,\b x) 
\text{ for all } t \in \R\,, x \in \R^3. \]

To derive the integral equation, we use that $\b E-\b E_{\textrm s}$, $\b B-\b B_{\textrm s}$ is a solution of \eqref{eqMaxwellMacro} with $\bar\rho$ replaced by $\bar\rho-\rho_{\textrm s}$ and $\j$ by $\j-\b j_{\textrm s}$. Therefore, according to the representation~\eqref{eqElectricFieldSol} of the solution derived in \autoref{thElectricFieldSol}, we find for $\b E$ that
\begin{align*}
\b E(t,\b x)&=\b E_{\textrm s}(\b x)+\b E^{(0)}(t,\b x)-\frac1{c^2t}\int_{\partial B_{ct}(\b x)}(\j(0,\b y)-\b j_{\textrm s}(\b y))\d s(\b y) \\
&-\int_{B_{ct}(0)}\frac1{|\b y|}\left(\grad_{\b x}\bar\rho(t-\tfrac{|\b y|}c,\b x+\b y)-\grad\rho_{\textrm s}(\b x+\b y)+\frac1{c^2}\partial_t \j(t-\tfrac{|\b y|}c,\b x+\b y)\right)\d \b y.
\end{align*}
Using that $\j(t,\b x)=\b j_{\textrm s}(\b x)$ for $t\le0$ and $\bar\rho(t,\b x)=\rho_{\textrm s}(\b x)-\int_{-\infty}^t\div_{\b x}\bar j(\tau,\b x)\d\tau$, we find that
\begin{align*}
\b E(t,\b x)-\b E_{\textrm s}(\b x) &= \b E^{(0)}(t,\b x)+\int_{\R^3}\int_{-\infty}^{t-\frac{|\b y|}c}\frac1{|\b y|}\grad_{\b x}\div_{\b x}\j(\tau,\b x+\b y)\d\tau\d \b y \\
&-\frac1{c^2}\int_{\R^3}\frac1{|\b y|}\partial_t \j(t-\tfrac{|\b y|}c,\b x+\b y)\d \b y.
\end{align*}
\end{proof}

\subsection{Measurements}
To simplify the calculations, we assume that we the medium satisfies $\b E_{\textrm s} =0$ and $\b B_{\textrm s}=0$ 
in a domain $\Omega$.

Next, we want to model the measurements of optical coherence tomography. We consider illuminating waves of 
the form
\begin{equation}
 \label{eq:ill}
 \b E^{(0)}(t,\b x)=f(t+\tfrac{ x_3}c)\bm\eta
\end{equation}
with a function $f\in C^\infty_{\mathrm c}(\R)$ and a fixed polarisation vector $\bm\eta\in\R^2\times\{0\}$, where we choose 
$f$ such that $\b E^{(0)}$ fulfils the condition \eqref{eqIntersect}. 

Moreover, we place point detectors everywhere on the detector surface 
\begin{equation}
\label{eq:surface}
\mathcal D=\R^2\times\{d\}
\end{equation}
with $d>0$ sufficiently large, in particular, we require that $\b E^{(0)}(t,\b\xi)=0$ for $t\ge0$ and $\b\xi\in\mathcal D$. For more details, we refer to \cite[Equation (29)]{ElbMinSch15}.

Then, the reference wave $\b E^z$ produced by reflecting this incoming wave at a perfect mirror placed in the plane given by the equation $x_3=z$ is given by
\begin{equation}\label{eqReflectedField}
\b E^z(t,\b x) = \begin{cases}\big(f(t+\frac{x_3}c)-f(t+ \frac{x_3}c + 2\,\frac{z-x_3}c) \big)\bm\eta&\text{if}\;x_3>z,\\0&\text{if}\;x_3\le z.\end{cases}
\end{equation}
The measurements done in optical coherence tomography are now the intensities
\begin{equation}\label{eqIntensity}
I_j(z,\b\xi)=\int_0^\infty|E_j(t,\b\xi)+E^z_j (t,\b\xi)|^2\d t,\quad\b\xi\in\mathcal D,\;j\in\{1,2,3\},
\end{equation}
on the detector surface $\mathcal D$.

We combine these measurements to
\begin{equation}\label{eqEffectiveMeasurements}
\tilde I_j(z,\b\xi) =  \frac12\left(I_j(z,\b\xi)-\int_0^\infty|E_j(t,\b\xi)|^2\d t -\int_0^\infty|E^z_j (t,\b\xi)|^2\d t\right),
\end{equation}
where the second term can be obtained by measuring the back-scattered wave without superimposing the reflected wave $\b E^z$, and the last term is explicitly known from the initial laser pulse~$\b E^{(0)}$. Equation \eqref{eqEffectiveMeasurements} results to
\[
\tilde I_j(z,\b x) = \int_{0}^\infty E_j (t,\b x) E^z_j (t,\b x)\d t .
\]
Since we have by our choice of measurement setup that $\b E(t,\b x)=\b E^{(0)}(t,\b x)$ for $t\le0$ and $\b E^{(0)}(t,\b\xi)=0$ for $t>0$ and $\b\xi\in\mathcal D$, we can write this in the form
\begin{align*}
\tilde I_j(z,\bm\xi) &= \int_{0}^\infty(E_j-E_j^{(0)})(t,\bm\xi)(E^z_j -E_j^{(0)})(t,\bm\xi)\d t \\
&= \int_{-\infty}^\infty(E_j-E_j^{(0)})(t,\bm\xi)(E^z_j -E_j^{(0)})(t,\bm\xi)\d t.
\end{align*}
Inserting the explicit formula \eqref{eqReflectedField} for the reflected wave $\b E^z$, we get
\begin{equation}\label{eq:ift}
\tilde I_j(z,\bm\xi) = -\eta_j\int_{-\infty}^\infty(E_j-E_j^{(0)})(t,\bm\xi)f(t + \tfrac{2z-\bm\xi_3}c)\d t.
\end{equation}

We use the convention
\begin{equation}
 \label{eq:fourier}
 \mathcal FF(\omega) = \int_{-\infty}^\infty F(t)\e^{\i\omega t}\d t
\end{equation}
for the Fourier transform of an integrable function $F\in L^1(\R)$ with respect to time, and we put a subindex at $\mathcal F$ if we need to specify the variable with respect to which we do the Fourier transform.

Taking the inverse Fourier transform with respect to $z$ in \eqref{eq:ift} and using Plancherel's formula we get
\[
 \frac2c \int_{-\infty}^\infty \tilde I_j(z,\bm\xi) \e^{-\i \tfrac{\omega}c (2z- \bm\xi_3)} \d z = -\eta_j (\mathcal F_t E_j-\mathcal F_t E_j^{(0)})(\omega,\bm\xi) \mathcal F f(-\omega ).
\]
Provided that the Fourier transform of the function $f$ is nowhere zero and that the polarisation~$\bm\eta$ fulfils $\eta_1\ne0$ and $\eta_2\ne0$, we obtain from this the data
\[ \b h(t,\bm\xi) = \begin{pmatrix}E_1(t,\bm\xi)\\E_2(t,\bm\xi)\end{pmatrix}, \]
where $\b h$ can be directly calculated from the measurements $\tilde I_j$ and the knowledge of the initial wave $\b E^{(0)}$:
\begin{equation}\label{eqh}
\mathcal F_t h_j (\omega ,\bm\xi) = \mathcal F_t E_j^{(0)}(\omega,\bm\xi) - \frac2{c \eta_j \mathcal F f(-\omega )} \int_{-\infty}^\infty \tilde I_j(z,\bm\xi) \e^{-\i \tfrac{\omega}c (2z- \bm\xi_3)} \d z .
\end{equation}

The inverse problem of optical coherence tomography is now to find the material parameter $\mu:\R\times\R^3\to\R$ from the function $\b h:\R\times\mathcal D\to\R^2$ according to the system of equations 
\begin{equation}\label{eqOCT}
\boxed{
\begin{aligned}
\frac1{c^2}\partial_{tt} \b E(t,\b x)+\curl_{\b x}\curl_{\b x} \b E(t,\b x)&=-\frac{4\pi}{c^2}\int_{-\infty}^\infty\partial_t\mu(\tau,\b x)\b E(t-\tau,\b x)\d\tau,\quad&&t>0,\;\b x\in\R^3, \\
\b E(t,\b x) &= f(t+\tfrac{\b x_3}c)\bm\eta\quad&&t\le0,\;\b x\in\R^3, \\
\b h(t,\bm\xi) &= \begin{pmatrix}E_1(t,\bm\xi)\\E_2(t,\bm\xi)\end{pmatrix},\quad&&t>0,\;\bm\xi\in\mathcal D.
\end{aligned}}
\end{equation}
At first, it seems that we have enough data to solve this problem: we have the three dimensional function $\b h$ for every choice of function $f$ and polarisation $\bm\eta$. However, since the problem is linear in the electric field $\b E$, we do not gain any information by changing $f$ or having more than two linearly independent polarisations $\bm\eta$. To see this, we perform a Fourier transform $\mathcal F$ with respect to time. 

\begin{lemma}\label{thOCTLinearity}
Let $\mu:\R\times\R^3\to\R$ be given, $\bm\eta^{(k)}\in S^1\times\{0\}\subset\R^2\times\{0\}$, $k=1,2,3$, be so that the first two are linearly independent and let $f^{(k)}$ be sufficiently regular functions so that the first two functions $f^{(k)}$ have a non-vanishing Fourier transform:
\[ \mathcal Ff^{(k)}(\omega) \ne 0\quad\text{for every}\quad\omega\in\R,\;k\in\{1,2\}. \]

Moreover, let $\b E^{(k)}$, $k=1,2,3$, be the solutions of
\begin{equation*}\begin{aligned}
\frac1{c^2}\partial_{tt}\b E^{(k)}(t,\b x)+\curl_{\b x}\curl_{\b x} \b E^{(k)}(t,\b x)&=-\frac{4\pi}{c^2}\int_{-\infty}^\infty\partial_t\mu(\tau,\b x)\b E^{(k)}(t-\tau,\b x)\d\tau, &&t>0,\;\b x\in\R^3, \\
\b E^{(k)}(t,\b x) &= f^{(k)}(t+\tfrac{\b x_3}c)\bm\eta^{(k)}, &&t\le0,\;\b x\in\R^3.
\end{aligned}\end{equation*}

Then, the measurement data
\begin{equation}
 \label{eq:measurement_oct}
\b h^{(k)}(t,\bm\xi) = \begin{pmatrix}E^{(k)}_1(t,\bm\xi)\\E^{(k)}_2(t,\bm\xi)\end{pmatrix},\quad t>0,\;\bm\xi\in\mathcal D.
\end{equation}
fulfil the relation
\begin{equation}\label{eqOCTLinearity}
\boxed{
\mathcal F_t \b h^{(3)}(\omega,\bm\xi) = \sum_{k=1}^2 c_k\frac{\mathcal F f^{(3)}(\omega)}{\mathcal Ff^{(k)}(\omega)}\mathcal F_t \b h^{(k)}(\omega,\bm\xi),\quad \omega\in\R,\;\bm\xi\in\mathcal D,}
\end{equation}
where the coefficients $c_k\in\R$ are determined by $\bm\eta^{(3)}=\sum_{k=1}^2 c_k\bm\eta^{(k)}$.
\end{lemma}
\begin{proof}
As in \autoref{thElectricFieldSolMacro}, we consider  the equivalent integral equation \eqref{eqElectricFieldMacroIntegral} for $\b E_s = 0$ 
\begin{align*}
\b E^{(k)}(t,\b x) = f^{(k)}(t+\tfrac{\b x_3}c)\bm\eta^{(k)} -\int_{\R^3}\int_{-\infty}^{t-\frac{|\b y|}c}\frac1{|\b y|}\left(\grad_{\b x}\div_{\b x}-\frac1{c^2}\partial_{tt}\right)\j^{(k)}(\tau,\b x+\b y)\d\tau\d \b y.
\end{align*}
Applying a Fourier transform with respect to time to the above equation and using \eqref{eqDielectricum} in the frequency domain we get that
\begin{align}\label{eqOCTFourierDomain}
\mathcal F_t \b E^{(k)}(\omega,\b x) &=\mathcal Ff^{(k)}(\omega)\e^{-\i\frac\omega cx_3}\bm\eta^{(k)} \nonumber\\
&+\left(\grad_{\b x}\div_{\b x}+\frac{\omega^2} {c^2}\right)\int_{\R^3}\frac{\i\e^{\i\frac\omega c|\b x-\b y|}}{\omega|\b x-\b y|}\mathcal F_t\mu(\omega,\b y)\mathcal F_t \b E^{(k)} (\omega,\b y)\d \b y .
\end{align}
Now, the linearity of this equation implies that the solution $\mathcal F_t \b E^{(3)}$ of the third equation can be written in the form
\[ \mathcal F_t \b E^{(3)}(\omega,\b x) = \sum_{k=1}^2 c_k\frac{\mathcal Ff^{(3)}(\omega)}{\mathcal Ff^{(k)}(\omega)}\mathcal F_t \b E^{(k)}(\omega,\b x). \]
Restricting this relation to the detector surface $\mathcal D$, we get \eqref{eqOCTLinearity}.
\end{proof}

In particular, even if one considers the assumptions of Born and far field approximations (specified later), we see that the measurements are equivalent to the determination of the four dimensional Fourier transform of the function $\mu$ on a cone in $\R^3$, see \cite[Proposition 5.2]{ElbMinSch15}. Thus, also for OCT, we do not have sufficient data to uniquely recover the material parameter $\mu$. 

However, under certain simplifications there exist works that provide reconstructions. A main assumption is that the medium is non-dispersive, meaning that the temporal Fourier transform of $\mu$ does not depend on frequency. Under this assumption, Mark \textit{et al} \cite{MarRalBopCar07} proposed algorithms under the Born approximation, and in \cite{FerHitKamZai95, Hel96} the one dimensional case was considered. Also in \cite{RalMarKamBop05, RalMarCarBop06} the the OCT system is described using a single backscattering model.


\section{Combined System}\label{sePAT_OCT}
We have seen that the inverse problems of quantitative photoacoustic imaging and OCT both lack data to obtain a unique 
reconstruction of the material parameters depending on the specific modeling. Since the illumination for photoacoustic imaging and OCT are modeled analogously, the rely on the same  electromagnetic material parameters of the medium. Therefore, we are suggesting to use the two measurements in combination 
to obtain additional informations.

To this end, we want to rewrite the optical problem in photoacoustics in terms of the material parameter $\mu$ \eqref{eqDielectricum} 
used in Maxwell's equations for the modelling of OCT instead of the absorption and scattering coefficients appearing in the radiative transfer equation. Thus, we need an expression of the absorbed energy of an electromagnetic wave in an dielectric medium.

Below we summarize the variables, which are used throughout this section, as a backup reference:
\begin{longtable}[c]{|l|l|c|}
\hline
Symbol & Quantity & Definition \\ \hline
 $\bm\eta$ & incident polarization vector & \eqref{eq:ill}\\
 ${\bm E}^{(0)}$ & incident field & \eqref{eq:ill}\\
   $\mathcal{D}$ & detector's array in OCT & \eqref{eq:surface}\\
 ${\bm E}^z$ & reflected field in OCT & \eqref{eqReflectedField}\\
  $I_j$ & measured intensity in OCT & \eqref{eqIntensity}\\
 $\tilde I_j$ & effective measured intensity in OCT & \eqref{eqEffectiveMeasurements}\\
 $\b h$ & OCT measurements & \eqref{eqh} \\
   $p^{(0)}$ & PAT measurements  & \eqref{eqCombinedPAT} \\
 $\tilde h$ & modified OCT measurements & \eqref{defH} \\
    $\tilde p$ & modified PAT measurements  & \eqref{eqCombinedPATSimplified} \\
    \hline
\caption{Combined PAT-OCT System}
\end{longtable}
\subsection{Absorbed energy}
We therefore start again with Maxwell's microscopic equations \eqref{eqMaxwell}, \eqref{eqChargeAndCurrent} and \eqref{eqLorentzForce} to describe the medium and define the averaged mechanical energy as the sum of the kinetic energies of the particles and their electrostatic potential energy:
\begin{equation}
\label{eq:me}
 \mathcal E(t,\b x) = \sum_{i=1}^N\frac{m_i}2|\b x_i'(t)|^2 w(\b x_i(t)-\b x)+\frac12\sum_{i\ne j}\frac{q_iq_j}{|\b x_i(t)-\b x_j(t)|}w(\b x_i(t)-\b x)w(\b x_j(t)-\b x)\;.
\end{equation}
Assuming now that no particles enter the support of $w(\cdot-\b x)$ or leave the region where $w$ is constant, the change in energy is given by
\begin{equation} \label{eq:Et}
 \begin{aligned}
\partial_t\mathcal E(t,\b x) &= \sum_{i=1}^Nm_i\left<\b x_i'(t),\b x_i''(t)\right>w(\b x_i(t)-\b x) \\
&-\frac12\sum_{i\ne j}q_iq_j \frac{\left<\b x_i(t)-\b x_j(t),\b x_i'(t)-\b x_j'(t)\right>} {|\b x_i(t)-\b x_j(t)|^3}w(\b x_i(t)-\b x)w(\b x_j(t)-\b x)\\
&= \sum_{i=1}^Nm_i\left<\b x_i'(t),\b x_i''(t)\right>w(\b x_i(t)-\b x) \\
&-\sum_{i\ne j}q_iq_j\frac{\left<\b x_i(t)-\b x_j(t),\b x_i'(t)\right>}{|\b x_i(t)-\b x_j(t)|^3}w(\b x_i(t)-\b x)w(\b x_j(t)-\b x).
\end{aligned}
\end{equation}
We consider the limit when the charge distributions $\rho_i$ of the single particles tend to $\delta$-distributions, 
then the Lorentz force \eqref{eqLorentzForce} is given by
\begin{equation*}
 m_i \b x_i'' (t) = q_i \b e(t,\b x_i(t)) + q_i \frac{\b x_i'(t)}{c} \times \b b(t,\b x_i(t))= q_i \b e(t,\b x_i(t)) + q_i \frac{\b x_i'(t)}{c}\,,
\end{equation*}
where the last identity is due to the fact that $\tfrac{\b x_i'}{c} \times \b b$ is orthogonal to $\b x_i'$.

Using this identity \eqref{eq:Et} we find that
\begin{align*}
\partial_t\mathcal E(t,\b x) &= \sum_{i=1}^Nq_i\left<\b x_i'(t),\b e(t,\b x_i(t))\right>w(\b x_i(t)-\b x) \\
&-\sum_{i\ne j}q_iq_j\frac{\left<\b x_i(t)-\b x_j(t),\b x_i'(t)\right>}{|\b x_i(t)-\b x_j(t)|^3}w(\b x_i(t)-\b x)w(\b x_j(t)-\b x).
\end{align*}

We further assume that the particles are moving relatively slow so that we may approximate the electric field $\b e$ by the electrostatic approximation, which can be represented via \eqref{eqCompact} by neglecting the contribution of the current $\b j$:
\begin{align*}
 \b e(t,\b x) &\approx \b e^{(0)}(t,\b x)-\int_{B_{ct}(0)}\frac1{|\b y|}\grad_{\b x}\rho(t,\b x+\b y)\d \b y \\
 &\approx \b e^{(0)}(t,\b x) - \sum_{j=1}^Nq_j \int_{\R^3} \frac1{|\b y|}\grad_{\b x} \delta (\b x+\b y - \b x_j(t))\d \b y  \\
&\approx \b e^{(0)}(t,\b x)+\sum_{j=1}^Nq_j\frac{\b x-\b x_j(t)}{|\b x-\b x_j(t)|^3}.
\end{align*}
where $\b e^{(0)}$ denotes the solution of the homogeneous wave equation with initial data $\b e^{(0)}(0,\b x)=\b e(0,\b x)$ and $\partial_t \b e^{(0)}(0,\b x)=\curl_{\b x} \b b(0,\b x)$. Here, we have used that for $t$ sufficiently large, $B_{ct}(0)$ will contain the complete medium and \eqref{eqChargeDensity} where $\rho_i$ is replaced by a $\delta$-distribution.

Then, we can write the change in energy in the form
\begin{equation}
\begin{aligned}
\label{eq:pt}
\partial_t\mathcal E(t,\b x) &\approx \sum_{i=1}^Nq_i\left<\b x_i'(t),\b e^{(0)}(t,\b x_i(t))\right>w(\b x_i(t)-\b x) \\
&+\sum_{i\ne j}q_iq_j\frac{\left<\b x_i(t)-\b x_j(t),\b x_i'(t)\right>}{|\b x_i(t)-\b x_j(t)|^3}w(\b x_i(t)-\b x)(1-w(\b x_j(t)-\b x)).
\end{aligned}
\end{equation}

Until now, we had no particular assumptions on the size of the support of $w$. We have seen that the particles have diameter of order \SI{e-14}{\metre}. The size of molecules in biological tissue is of order \SI{e-10}{\metre}, a water molecule for example has a 
diameter of \SI{0.3}{\nano\metre}. Thus, we specify $r$ to be of order \SI{e-10}{\metre}.
In what follows we assume in addition to $w\in C^\infty_{\mathrm c}(\R^3)$ that 
\begin{equation}
\label{eq:w}
 w \equiv \text{const in } B_r(0) \text{ and } w \equiv 0 \text{ outside } B_{r+\varepsilon}(0)\;.
\end{equation}
Practically, the typical wavelength of photoacoustic imaging and OCT is around \SI{1000}{\nano\metre}, which is three orders of magnitude larger than the size of molecules. Then, a ball with radius $\tilde{r}$ of order \SI{e-6}{\metre} will denote the resolution limitation and we can assume that $\tilde{r}\gg r.$

For the averaging here, we assume that $\b e^{(0)}$ is almost constant on the support of $w(\cdot-\b x)$. Then, we may approximate 
\begin{equation}
 \b e^{(0)}(t,\b x_i(t)) \approx \int_{\R^3} \b e^{(0)}(t,\b y)w(\b y-\b x)\d \b y, \quad \b x_i(t)\in B_r (\b x)\;.
\end{equation}
The second term in \eqref{eq:pt} is nonzero only when $\b x_i \in B_r (\b x)$ and $\b x_j \notin B_r (\b x) .$ We want to assume that locally (related to resolution limit) there is almost no energy exchange:
\[
\sum_{i\ne j}q_iq_j\frac{\left<\b x_i(t)-\b x_j(t),\b x_i'(t)\right>}{|\b x_i(t)-\b x_j(t)|^3}w(\b x_i(t)-\b x)(1-w(\b x_j(t)-\b x)) \approx 0, \quad \b x_j (t)\in B_{\tilde{r}} (\b x) \setminus B_r (\b x) .
\]

For the case, where $\b x_i \in B_r (\b x)$ and $\b x_j \notin B_{\tilde{r}} (\b x),$ since $\tilde{r}\gg r$, 
we can approximate $\b x_i - \b x_j \simeq \b y - \b x_j ,$ for every $\b y \in B_r (\b x)$ 
and we thus obtain
\begin{align*}
\partial_t\mathcal E(t,\b x) &\approx \left<\int_{\R^3}\b e^{(0)}(t,\b y)w(\b y-\b x)\d \b y,\sum_{i=1}^Nq_i \b x_i'(t)w(\b x_i(t)-\b x)\right> \\
&+\left<\sum_{j=1}^Nq_j\int_{\R^3}\frac{\b y-\b x_j(t)}{|\b y-\b x_j(t)|^3}w(\b y-\b x)\d \b y(1-w(\b x_j(t)-\b x)),\sum_{i=1}^Nq_i \b x_i'(t)w(\b x_i(t)-\b x)\right>.
\end{align*}
Now, since the support of $w$ is small compared to the medium, we may take in the second term the sum over all particles $\b x_j$ without introducing a large error and find that
\[ \partial_t\mathcal E(t,\b x) \approx \left< \b E(t,\b x),\j(t,\b x)\right>. \]

Thus, in a non-magnetic, isotropic, linear dielectric, we have that the total absorbed energy is given by
\begin{equation}\label{eqAbsorbedEnergy}
\lim_{t\to\infty}\mathcal E(t,\b x)-\lim_{t\to-\infty}\mathcal E(t,\b x) \approx \int_{-\infty}^\infty\int_{-\infty}^\infty\mu(\tau,\b x)\left<\b E(t,\b x),\b E(t-\tau,\b x)\right>\d\tau\d t.
\end{equation}

\begin{remark}\label{rem41}
A different derivation of the absorbed energy follows from the definition of the total electromagnetic energy density and the energy transport of the electromagnetic wave (Poynting vector) as follows:
\begin{equation}
 \label{eq:ws}
 \mathcal W = \frac1{8\pi} ( \langle \b E , \b D  \rangle+ \langle \b B , \b B \rangle ), \quad \text{and} \quad
 \b S =\frac{c}{4\pi} ( \b E\times \b B )\,,
\end{equation}
respectively. Here, 
\begin{equation}
 \label{eq:electric_displacement}
 \b D = \b E + 4\pi \b P
\end{equation}
denotes the electric displacement. Using the elementary relation $\div_{\b x} (\b F \times \b G) = \b G \cdot \curl_{\b x} \b F - \b F \cdot \curl_{\b x} \b G$, 
the Maxwell's equations \eqref{eqMaxwellMacro3} and \eqref{eqMaxwellMacro4} it follows that
\[
 \div_{\b x} \b S = -\frac{1}{4\pi} \left( \langle \b B, \partial_t \b B \rangle +
 \langle \b E, \partial_t \b E \rangle \right) -  \langle \b E,\bar{\b j} \rangle .
\]
 Then, from the definition of $\mathcal W,$ integrating the above equation over a subvolume $\Omega$ of the medium and applying the divergence theorem gives 
\begin{equation}\label{conservationenergy2}
\begin{aligned}
- \partial_t \int_\Omega \mathcal W (t,\b x) \d \b x 
&= \int_{\partial \Omega} \langle \b S (t,\b x) , \b N \rangle \d s (\b x) + 
   \int_\Omega \langle \b E (t,\b x),\bar{\b j} (t,\b x) \rangle \d \b x \\
   &- \frac12 \int_\Omega \partial_t \langle \b E (t,\b x),\b P (t,\b x) \rangle \d \b x
\end{aligned}
\end{equation}
where $\b N$ is the outward pointing unit normal vector. 

Equation \eqref{conservationenergy2} describes the balance of energy:  By the definition of $\mathcal W$, 
the left hand side quantifies the variation of the total energy in the volume. 
The first term on the right hand side describes the variation of the energy across the boundary, and the last two terms describe the dissipated power. According to \cite{Jac98}, the integrands of the last two terms in the right hand side integrated over time denotes the electric energy density, which reads as follows
\begin{equation*}
\begin{aligned}
  \int_{-\infty}^{\infty} \langle \b E (t,\b x),\bar{\b j} (t,\b x)\rangle \d t &= 
  \int_{-\infty}^{\infty} \int_{-\infty}^{\infty} \mu(\tau,x) \langle \b E(t,\b x),\b E(t-\tau,\b x) \rangle \d\tau\d t\,,
\end{aligned}
\end{equation*}
resulting again in \eqref{eqAbsorbedEnergy}, if we assume that the polarization goes to zero after a long time.
\end{remark}

\subsection{Inverse problem}
Using the expression~\eqref{eqAbsorbedEnergy} for the absorbed energy as a replacement of $\mu_{\mathrm a}(\b x)\bar\Phi(\b x)$ in \eqref{eqInitialPressure}, we get by combining~\eqref{eqInitialPressure} and~\eqref{eqOCT} the inverse problem of calculating $\mu$ and $\gamma$ from the measurement data $p^{(0)}$ and $h$ for giving initial pulse $f$ and polarisation $\bm\eta$ obeying the equation system
\begin{equation*}\begin{aligned}
\frac1{c^2}\partial_{tt}\b E(t,\b x)&=-\curl_{\b x}\curl_{\b x} \b E(t,\b x)-\frac{4\pi}{c^2}\int_{-\infty}^\infty\partial_t\mu(\tau,\b x) \b E(t-\tau,\b x)\d\tau,\quad&&t>0,\;\b x\in\R^3, \\
\b E(t,\b x) &= f(t+\tfrac{\b x_3}c)\bm\eta,\quad&&t\le0,\;\b x\in\R^3, \\
p^{(0)}(\b x) &= \gamma(\b x)\int_{-\infty}^\infty\int_{-\infty}^\infty\mu(\tau,\b x)\left<\b E(t,\b x),\b E(t-\tau,\b x)\right>\d\tau\d t,\quad&&\b x\in\R^3, \\
\b h(t,\bm\xi) &= \begin{pmatrix}E_1(t,\bm\xi)\\E_2(t,\bm\xi)\end{pmatrix},\quad&&t>0,\;\bm\xi\in\mathcal D.
\end{aligned}\end{equation*}
Now, although varying the initial pulse $f$ does not give us additional information with respect to the measurements of OCT, see \autoref{thOCTLinearity}, its influence in the photoacoustic part is non-linear and provides us with independent data. Thus, by choosing a one-parameter family of functions $f$, for example almost monochromatic waves $f_\nu$ centered at a frequency $\nu$ for arbitrary $\nu>0$, we get for every $\nu>0$ the equation system
\begin{subequations}\label{eqCombined}
\begin{empheq}[box=\fbox]{align}
\frac1{c^2}\partial_{tt} \b E_\nu &= -\curl_{\b x}\curl_{\b x} \b E_\nu-\frac{4\pi}{c^2}\int_{-\infty}^\infty\partial_t\mu(\tau,\b x) \b E_\nu(t-\tau,\b x)\d\tau, & t>0,\;\b x\in\R^3, \label{eqCombinedWave} \\
 \b E_\nu(t,\b x) &= f_\nu(t+\tfrac{\b x_3}c)\bm\eta, & t\le0,\;\b x\in\R^3, \label{eqCombinedWaveInitial} \\
p^{(0)}_\nu(\b x) &= \gamma(\b x)\int_{-\infty}^\infty\int_{-\infty}^\infty\mu(\tau,\b x)\left<\b E_\nu(t,\b x),\b E_\nu(t-\tau,\b x)\right>\d\tau\d t, & \b x\in\R^3, \label{eqCombinedPAT} \\
\b h_\nu(t,\bm\xi) &= \begin{pmatrix}E_{\nu,1}(t,\bm\xi)\\E_{\nu,2}(t,\bm\xi)\end{pmatrix}, & t>0,\;\bm\xi\in\mathcal D. \label{eqCombinedOCT}
\end{empheq}
\end{subequations}
Now, the first two equations uniquely determine $\b E_\nu$ as a function of $\mu$. Then, depending on~$\gamma$, we would have to solve the third equation for the four dimensional function $\mu$, which would leave us with two three dimensional equations for the three dimensional Gr\"uneisen parameter $\gamma$.

Although the dimensions of the data seem to perfectly fit for a reconstruction, it is still an open problem if these equations uniquely determine the material parameters, although there are partial results in this direction, see~\cite{BalZho14}, where the injectivity of a very similar system is discussed.

As a plausibility argument, we want to show in the next section that at least in a simplified case, the parameters can be uniquely recovered.

\section{Weakly Scattering Medium and multiple Pulsed Laser Illuminations
}\label{seFre}

The full system \eqref{eqCombined} can be equivalently transformed to a Fredholm integral equation of the second kind under the following assumptions:
\begin{enumerate}
\item Born approximation: The medium is weakly scattering, meaning that $\mathcal F_t\mu$ is sufficiently small.
\item Far-field approximation: Typically, in an OCT system, the measurements are performed in a distance much bigger compared to the size of the medium. 
\item Specific illumination: The support of the initial pulses tends to a delta-distribution such that the optical parameter can be assumed constant in this spectrum.
\end{enumerate}

In a first step, we derive a Fredholm integral equation of the first kind for the unknown Gr\"uneisen parameter.
\begin{proposition}
The complete system \eqref{eqCombined} under the above three assumptions is reduced to the integral equation 
\begin{equation} 
\label{eq:ig1}
 \int_{\R^3}K[\tilde p](\nu,\bm\vartheta;\b y)\frac1{\gamma(\b y)}\d \b y = \tilde h(\nu,\bm\vartheta) 
\end{equation}
where
\begin{equation} 
\label{eq:ig2}
 K[\tilde p](\nu,\bm\vartheta;\b y) = \left(\tilde p(\nu,\b y)-
 \frac\i\pi\int_{-\infty}^\infty\frac{\tilde p(\tilde\nu,\b y)}{\tilde\nu-\nu}\d\tilde\nu\right)
 \e^{-\i\frac\nu c\left<\bm\vartheta+\b e_3,\b y\right>} , \quad \tilde p (\nu, \b x) :=2\pi p_\nu^{(0)} (\b x)
\end{equation}
and
\begin{equation}\label{defH}
\tilde h(\nu,\bm\vartheta) = \sum_{j=1}^2 \left( \frac{\mathcal F_t \b E_\nu^{(b)}(\nu,R\bm\vartheta)}{\mathcal Ff_\nu(\nu)} - \e^{-\i\frac\nu cR\vartheta_3} \eta_j \right) \frac{\i Rc^2 e^{-\i \frac\nu c R}}{\nu ( \vartheta\times\vartheta\times\eta)_j} .
\end{equation}
\end{proposition}

\begin{proof}
We start with \eqref{eqCombinedWave} - \eqref{eqCombinedWaveInitial}. We recall equation \eqref{eqOCTFourierDomain} derived in the proof of \autoref{thOCTLinearity}, then we obtain
\begin{align*}
\mathcal F_t \b E_\nu(\omega,\b x) &=\mathcal Ff_\nu(\omega)\e^{-\i\frac\omega cx_3}\bm\eta +\left(\grad_{\b x}\div_{\b x}+\frac{\omega^2} {c^2}\right)\int_{\R^3}\frac{\i\e^{\i\frac\omega c|\b x-\b y|}}{\omega|\b x-\b y|}\mathcal F_t\mu(\omega,\b y)\mathcal F_t \b E_\nu(\omega,\b y)\d \b y ,
\end{align*}
that holds for every frequency $\nu >0.$ Under the Born approximation, we replace $\mathcal F_t \b E_\nu$ in the integral by the incident field, to obtain $\b E_\nu^{(b)}$:
\begin{align*}
\mathcal F_t \b E_\nu^{(b)}(\omega,\b x) =\mathcal Ff_\nu(\omega)\left(\e^{-\i\frac\omega cx_3}\bm\eta+\left(\grad_{\b x}\div_{\b x}+\frac{\omega^2}{c^2}\right)\int_{\R^3}\frac{\i\e^{\i\frac\omega c(|\b x-\b y|- y_3)}}{\omega|\b x-\b y|}\mathcal F_t\mu(\omega,\b y)\d \b y\,\bm\eta\right).
\end{align*}
We set $\b x = R\bm\vartheta, R>0$ and $\bm\vartheta \in \mathbbm{S}^2.$  Considering now the far-field approximation, we can replace the above expression by its asymptotic behaviour for $R\to\infty,$ uniformly in $\bm\vartheta.$ 
This gives us, see for example \cite[Equation (4.1) in Chapter 4.1]{ElbMinSch15}:
\begin{equation}\label{eqElectricFarBorn}
 \mathcal F_t \b E_\nu^{(b)}(\omega,R\vartheta)\simeq\mathcal Ff_\nu(\omega)
 \left(\e^{-\i\frac\omega cR\vartheta_3}\bm\eta-\frac{\i\omega e^{\i \frac\omega c R}}{Rc^2}
 \bm\vartheta\times\bm\vartheta\times\bm\eta
 \int_{\R^3}\e^{-\i\frac\omega c\left<\bm\vartheta+\b e_3,\b y\right>}
 \mathcal F_t\mu(\omega,\b y)\d \b y\right).
\end{equation}
Then, for every $\bm\vartheta\in\tilde{\mathcal D}=\{\bm\theta\in \mathbbm{S}^2\mid \exists R>0: R\bm\theta\in\mathcal D\}$ and $\nu \in \R^+$ the above equation using equation \eqref{eqCombinedOCT} and the definition \eqref{defH} can be rewritten as
\begin{equation}\label{eqCombinedOCTSimplified}
\tilde h(\nu,\bm\vartheta) = \int_{\R^3}\e^{-\i\frac\nu c\left<\bm\vartheta+\b e_3,\b y\right>}\mathcal F_t\mu(\nu,\b y)\d \b y,
\end{equation}
where we assumed that $\mathcal Ff_\nu(\nu)\ne0$. 

Similarly, we proceed with the equation \eqref{eqCombinedPAT} of the PAT measurement, written as
\begin{equation*}
\boxed{
 p^{(0)}_\nu(\b x) = \gamma(\b x)\int_{-\infty}^\infty \left<\b E_\nu(t,\b x),\int_{-\infty}^\infty\mu(\tau,\b x)\b E_\nu(t-\tau,\b x)\d\tau \right>\d t .}
\end{equation*}
From Plancherel's formula, if follows that
\begin{equation*}
\begin{aligned}
p^{(0)}_\nu(\b x) &= \gamma(\b x)\frac1{2\pi}\int_{-\infty}^\infty \left<\mathcal F_t \b E_\nu(\omega ,\b x),\overline{\mathcal F_t\mu(\omega,\b x)\mathcal F_t \b E_\nu(\omega,\b x)} \right>\d \omega \\
&= \gamma(\b x)\frac1{2\pi}\int_{-\infty}^\infty \overline{\mathcal F_t\mu(\omega,\b x)} |\mathcal F_t \b E_\nu(\omega,\b x)|^2 \d \omega \\
&= \gamma(\b x)\frac1{2\pi}\int_{-\infty}^\infty \mathcal F_t\mu(\omega,\b x) |\mathcal F_t \b E_\nu(\omega,\b x)|^2 \d \omega .
\end{aligned}
\end{equation*}
For the last equation we used that $p^{(0)}_\nu$ is real valued.
Considering again the Born approximation,
we obtain
\begin{equation*}
\boxed{
\tilde p(\nu,\b x)=\gamma(\b x)\int_{-\infty}^\infty\mathcal F_t\mu(\omega,\b x)|\mathcal Ff_\nu(\omega)|^2 \d\omega.
}
\end{equation*}
To take advantage of the third assumption, we choose initial pulses so that $\mathcal Ff_\nu$ has only support on a domain 
$\{\omega\in\R\mid |\omega|\in[\nu-\varepsilon,\nu+\varepsilon]\}$ with a sufficiently small $\varepsilon>0$ such that $\mathcal F_t\mu(\cdot,\b x)$ can be assumed to be for every $\b x\in\R^3$ constant on this support (we remark that since $f$ and $\mu$ are a real-valued functions, the real and imaginary parts of their Fourier transforms have to be even and odd, respectively). If we additionally normalise $\int_{-\infty}^\infty|\mathcal Ff_\nu(\omega)|^2\d\omega=\frac12$, we find that
\begin{equation}\label{eqCombinedPATSimplified}
\tilde p(\nu,\b x)=\gamma(\b x)\Re(\mathcal F_t\mu(\nu,\b x)).
\end{equation}
Moreover, since $\mu$ is a real-valued function, we know that its Fourier transform fulfils the Kramers--Kronig relation \cite{LanLif87}
\begin{equation}\label{eqKramersKronig}
\Im(\mathcal F_t\mu(\omega,\b x)) = -\frac1\pi\int_{-\infty}^\infty\frac{\Re(\mathcal F_t\mu(\tilde\omega,\b x))}{\tilde\omega-\omega}\d\tilde\omega.
\end{equation}

Thus, combining \eqref{eqCombinedOCTSimplified}, \eqref{eqCombinedPATSimplified}, and \eqref{eqKramersKronig}, we get the integral equation \eqref{eq:ig1} for the function $\frac1\gamma$ where the kernel depends on the data $\tilde p$ obtained from the PAT measurements.
\end{proof}

\begin{remark}
\begin{enumerate}
\item Compared to \cite[Formula (20)]{ElbMinSch15}, we have to mention that the different term in front of the integral in equation \eqref{eqElectricFarBorn} is due to the different scaling between the coefficients $\mu$ and $\chi,$ see the definition of $\b P.$
\item In practice, we do have equation \eqref{eqCombinedOCTSimplified} for every $\nu \in \R^+$ since the band-limited source allows measurements only for a fixed frequency spectrum.
\item In the most commonly used formulas appears the imaginary part of the permittivity or the conductivity in \eqref{eqCombinedPATSimplified}. The differentiation of our result is due to the definition of $\mu ,$ see \autoref{remark34}, and should not confuse the reader.
\end{enumerate}
\end{remark}

\subsection{Analytical formulas for special material}
The particular form of the kernel $K,$ see equation \eqref{eq:ig2}, motivated us to transform the integral equation \eqref{eq:ig1} of the first kind to one of the second kind.  More precisely, in the special case, where the object mainly consists of a single material with the frequency dependence given by some function $\alpha$ and a density distribution $\beta$, that is, if the initial pressure is of the form
\begin{equation*}
 \tilde p(\nu, \b x) = \alpha(\nu)\beta(\b x)+\varepsilon(\nu,\b x)
\end{equation*}
for some small function $\varepsilon$, we may write the integral equation \eqref{eq:ig1} considering \eqref{eq:ig2} as
\[ A(\nu)\int_{\R^3}\frac{\beta(\b y)}{\gamma(\b y)}\e^{-\i\frac\nu c\left<\bm\vartheta+\b e_3,\b y\right>}\d \b y+\int_{\R^3}K[\varepsilon](\nu,\bm\vartheta;\b y)\frac{\beta(\b y)}{\gamma(\b y)}\d \b y = \tilde h(\nu,\bm\vartheta), \]
where
\[ A(\nu) = \alpha(\nu)-\frac\i\pi\int_{-\infty}^\infty\frac{\alpha(\tilde\nu)}{\tilde\nu-\nu}\d\tilde\nu. \]
If we assume that $A(\nu)\ne0$, then we may rewrite this as  a Fredholm integral equation of the second kind for the Fourier transform 
\[ \Gamma(\b k) = \int_{\R^3}\frac{\beta(\b y)}{\gamma(\b y)}\e^{-\i\left<\b k,\b y\right>}\d \b y \]
of the function $\frac\beta\gamma$:
\[ \Gamma(\b k)+\frac1{A(\nu)}\int_{\R^3}\hat K(\b k;\bm\kappa)\Gamma(\bm\kappa)\d\bm\kappa = \hat h(\b k), \]
where the kernel $\hat K$ is for all values $\b k\in\R^3$ which are of the form $\b k=\frac\nu c(\bm\vartheta+\b e_3)$, $\vartheta\in\tilde{\mathcal D}$, defined by
\begin{equation*}
\hat K(\tfrac\nu c(\bm\vartheta+\b e_3);\bm\kappa) = \frac1{A(\nu)}\frac1{(2\pi)^3}\int_{\R^3}K[\varepsilon](\nu,\bm\vartheta;\b y)\e^{\i\left<\bm\kappa,\b y\right>}\d \b y
\end{equation*}
and $\hat h$ is similarly defined by
\begin{equation*}
 \hat h(\tfrac\nu c(\bm\vartheta+\b e_3)) = \frac1{A(\nu)}\tilde h(\nu,\bm\vartheta).
\end{equation*}
In particular, we see that for a sufficiently small function $\varepsilon$, the integral equation for $\Gamma$ can have at most one solution \cite{Kre89, Waz11}.

\section{Relation between the Material Parameters}

In this section we want to discuss the connection between the parameter $\mu$, defined in \eqref{eqDielectricum}, and the absorption and scattering coefficients, $\mu_{\mathrm a}$ and $\mu_{\mathrm s}$, appearing
in the radiative transfer equation \eqref{eq:rte}. 
Recently, there has been established the connection between the radiative transfer equation (vector and scalar) and equations 
in classical electromagnetism~\cite{Mis10,Rip11}. However, even under simplifying assumptions, these results do not lead to 
a direct connection between the material parameters $\mu$, $\mu_{\mathrm a}$ and $\mu_{\mathrm s}$ as the following remark shows. 

\begin{remark}
This remark is based on classical results \cite{BorWol99, TsaKonDin00} and follows \cite{Rip11}. Let the medium contain $N$ well separated scattering particles. We assume that there exist balls $B_k , \, k=1,...,N$ containing only one particle with large radius. We define by $\b N_k$ the unit outward normal vector on the boundary $\partial B_k .$ Let us denote by $\b F^+$ and $\b F^-$ the outgoing (scattered) and incoming (incident) field for a given particle, respectively.  If we neglect the interference effects between the incoming and outgoing fields, we can set as incident on the subdomain $B_k$ the field
\[
\b E_k^- = \sum_{\b x \in \Gamma_k}\b E, \quad \mbox{where} \quad \Gamma_k = \left\lbrace \b x \in \partial B_k \mid \left\langle \b S , \b N_k\right\rangle < 0\right\rbrace ,
\]
where $\b S$ is defined in \autoref{rem41}. Then, $\b E_k^+ = \b E -\b E_k^- .$ We redefine the time averaged Poynting vector (the real part of the complex Poynting vector)
\[
\b S^\alpha (\nu, \b x):= \frac{c}{8\pi} \Re (  \mathcal F_t \b E^\alpha (\nu,\b x)  \times \overline{\mathcal F_t \b B^\alpha (\nu,\b x )}), \quad \alpha = -,+ ,
\]
describing the average flux of the incident and the scattered field, respectively. Then, the averaged rate of the energy incident on $k$th particle is given by $S_k^- (\nu):= \tfrac{1}{B_k}\int_{B_k} |\b S^{-} (\nu,\b x)| \d \b x .$

Then, since every domain contains only one particle, the coefficients obtain the forms
\begin{align*}
\mu_{\mathrm a,k} (\nu) &= \frac{B_k}{S_k^- (\nu)}\int_{B_k} \int_{\R} \langle \b E (t,\b x),\bar{\b j} (t,\b x)\rangle \d t \d \b x, \quad k=1,...,N \\ \mu_{\mathrm s,k} (\nu) &= \frac{B_k}{S_k^- (\nu)} \int_{\partial B_k} \left\langle \b S^+ (\nu,\b x), \b N_k \right\rangle  \d s(\b x), \quad k=1,...,N.
\end{align*}

If the assumptions of the previous section still hold, then we can approximate the absorbed power as in the derivation of \eqref{eqCombinedPATSimplified} resulting to
\[
\mu_{\mathrm a,k} (\nu) \simeq \frac{B_k}{S_k^- (\nu)}\int_{B_k} \Re(\mathcal F_t\mu(\nu,\b x))\d \b x .
\]

Observing the above relations, we can identify the connection between the real part of the Fourier transform of $\mu$ and the absorption coefficient. But still arises the question if the obtained absorption and scattering coefficients are the ones that appear in the RTE, see \autoref{sePAT}. The fact that they depend on the incident average rate, results to an indirect dependence on the initial illumination characterizing the coefficients as non material parameters. If we omit the Born approximation, the dependence on the electric fields is stronger. In addition, the assumptions on the inner structure of the medium, make the above formulas applicable only in special cases.
\end{remark}
 

\section{Conclusions}
To our knowledge this paper is the first to use Maxwell's equations to model PAT. 
So far, the Maxwell's equations were only considered for modelling Thermoacoustic tomography (TAT) where low-frequency radiation is used. However, since recently the connection between the radiative transfer equation and classical electromagnetism has been established \cite{Mis10, Rip11} there is no need to differentiate from the proposed model considering the different modalities photo- or thermo-acoustics. Additionally, in our model the frequency dependence of the optical parameter is not neglected and thus the effect of the different incident frequencies to the medium are included. This modelling allows to develop a unifying model for OCT and PAT, and thus we can describe the combined setup.

\section*{Acknowledgements}
The work of OS has been supported by the Austrian Science Fund (FWF), Project P26687-N25 
(Interdisciplinary Coupled Physics Imaging).

\appendix

\section{Solution of Maxwell's equations}
If the charge density $\varrho$ and the electric current $\b j$ are known, then Maxwell's microscopic equations~\eqref{eqMaxwell} can be explicitly solved.
\begin{proposition}\label{thElectricFieldSol}
Let $\b e$ and $\b b$ be solutions of \eqref{eqMaxwell} for given functions $\rho$ and $\b j$. Then, $\b e$ has the form
\begin{align}\label{eqElectricFieldSol}
\b e(t,\b x) &= \frac1{4\pi ct}\int_{\partial B_{ct}(\b x)}\left(\curl_{\b x} \b b(0,\b y)-\frac{4\pi}c \b j(0,\b y)\right)\d s(\b y)+\partial_t\left(\frac1{4\pi c^2t}\int_{\partial B_{ct}(\b x)}\b e(0,\b y)\d s(\b y)\right) \nonumber\\
&-\int_{B_{ct}(0)}\frac1{|\b y|}\left(\grad_{\b x}\rho(t-\tfrac{|\b y|}c,\b x+\b y)+\frac1{c^2}\partial_t \b j(t-\tfrac{|\b y|}c,\b x+\b y)\right)\d \b y.
\end{align}
\end{proposition}
\begin{proof}
We combine the equations \eqref{eqMaxwell3} and \eqref{eqMaxwell4} to the vector wave equation
\[ \frac1{c^2}\partial_{tt}\b e(t,\b x)+\curl_{\b x}\curl_{\b x} \b e(t,\b x)=-\frac{4\pi}{c^2}\partial_t \b j(t,\b x). \]
Then, using the vector identity $\curl_{\b x}\curl_{\b x} \b e(t,\b x)=\grad_{\b x}\div_{\b x} \b e(t,\b x)-\Delta_x \b e(t,\b x)$, we can transform this with \eqref{eqMaxwell1} to the inhomogeneous wave equation
\[ \frac1{c^2}\partial_{tt} \b e(t,\b x)-\Delta_{\b x} \b e(t,\b x)=-4\pi \b f(t,\b x),\qquad \b f(t,\b x)=\grad_{\b x}\rho(t,\b x)+\frac1{c^2}\partial_t 
\b j(t,\b x). \]
Subtracting the solution $\b e^{(0)}$ of the homogeneous wave equation with the initial conditions $\b e^{(0)}(0,\b x)=\b e(0,\b x)$ and $\partial_t \b e^{(0)}(0,\b x)=\curl_{\b x} 
\b b(0,\b x)-\frac{4\pi}c \b j(0,\b x)$, we find that $\b e-\b e^{(0)}$ solves the inhomogeneous wave equation with zero initial data. So, by Duhamel's principle, see for example \cite[Chapter 2.4, Theorem 4]{Eva98}, the solution is given by
\[ \b e(t,\b x) = \b e^{(0)}(t,\b x)-\int_0^{ct}\frac1{ct-\zeta}\int_{\partial B_{ct-\zeta}(\b x)}\b f(\tfrac{\zeta}c,\b y)\d s(\b y) \d\zeta. \]
Combining the two integrals, we obtain
\begin{equation}\label{eqCompact}
 \b e(t,\b x) = \b e^{(0)}(t,\b x)-\int_{B_{ct}(\b x)}\frac1{|\b y-\b x|}\b f(t-\tfrac{|\b y-\b x|}c,\b y)\d \b y.
\end{equation}
Plugging in the explicit formula for the solution $\b e^{(0)}$ of the homogeneous wave equation, see for example \cite[Chapter 2.4, Theorem 2]{Eva98}, we arrive at~\eqref{eqElectricFieldSol}.
\end{proof}

\bibliographystyle{plain}
\bibliography{ElbMinSch15a}

\end{document}